\documentclass[12pt]{article}

\usepackage{amsmath,amssymb,latexsym,amsfonts,mathrsfs,graphicx}
\usepackage{algpseudocode,algorithm,algorithmicx,bm}
\usepackage{fullwidth}
\usepackage{placeins}
\usepackage{makecell}
\usepackage{multirow}
\usepackage{float}
\setcellgapes{4pt}

\algrenewcommand\algorithmicrequire{\textbf{Precondition:}} 
\algrenewcommand\algorithmicensure{\textbf{Postcondition:}}
\algrenewcommand\alglinenumber[1]{\footnotesize #1} 
\algnewcommand{\IIf}[1]{\State\algorithmicif\ #1\ \algorithmicthen} 
\algnewcommand{\EndIIf}{\unskip\ \algorithmicend\ \algorithmicif}

\usepackage{graphicx}
\usepackage{setspace}
\usepackage{program}
\usepackage{subfig}
\usepackage{color}

\renewcommand{\quad}{$~~~\;\;\;$}
\newtheorem{theorem}{Theorem}

\newtheorem{proposition}{Proposition}
\newtheorem{lemma}{Lemma}
\newtheorem{corollary}{Corollary}

\newcommand{\transp}{{^{\rm T}}}

\renewcommand{\R}{\mathbb{R}}

\newcommand{\matr}[1]{\begin{bmatrix} #1 \end{bmatrix}}    
\def\transp{^{\rm T}}
\newcommand{\Oh}{\mathcal O}

\newcommand{\J}{\mathcal{J}}
\newcommand{\ip}[2]{\left\langle #1 , #2 \right\rangle}    
\def\int{\mathrm{int}}

\providecommand{\newoperator}[3]{%
\newcommand*{#1}{\mathop{#2}#3}}
\newoperator{\argmax}{\mathsf{argmax}}{}
\newoperator{\argmin}{\mathsf{argmin}}{}

\allowdisplaybreaks
\usepackage[margin=2.5cm]{geometry}
\allowdisplaybreaks
\author{Javier Pe\~na\thanks{Tepper School of Business,
Carnegie Mellon University, USA, {\tt jfp@andrew.cmu.edu}}
\and  Negar Soheili\thanks{College of Business Administration,  University of Illinois at Chicago, USA, {\tt nazad@uic.edu }}
}
\title{Projection and rescaling algorithm for finding maximum support solutions to polyhedral conic systems}

\begin{document}
\maketitle

\abstract{
We propose a simple {\em projection and rescaling algorithm} that  finds 
 {\em maximum support} solutions to the pair of feasibility problems
\[
\text{find}  \; x\in L\cap\R^n_{+} \;\;\;\; \text{ and } \; \;\;\;\;
\text{find} \; \hat x\in L^\perp\cap\R^n_{+},
\]
where $L$ is a linear subspace in $\R^n$ and $L^\perp$ is its orthogonal complement. 
The algorithm complements a {\em basic procedure} that involves only projections onto $L$ and $L^\perp$ with a periodic {\em rescaling step.}  The number of rescaling steps and thus overall computational work performed by the algorithm are bounded above in terms of a condition measure of the above pair of problems.

Our algorithm is a natural but significant extension of a previous projection and rescaling algorithm that finds a solution to the full support problem
\[
\text{find}  \; x\in L\cap\R^n_{++}
\]
when this problem is feasible.  As a byproduct of our new developments, we obtain a sharper analysis of the projection and rescaling algorithm in the latter special case.
}

\section{Introduction.}
We propose a simple {\em projection and rescaling algorithm} that  finds 
 {\em maximum support} solutions to the pair of feasibility problems
\[
\text{find}  \; x\in L\cap\R^n_{+} \;\;\;\; \text{ and } \; \;\;\;\;
\text{find} \; \hat x\in L^\perp\cap\R^n_{+},
\]
where $L$ is a linear subspace in $\R^n$ and $L^\perp$ is its orthogonal complement. 
This maximum support problem is central in optimization as it subsumes any polyhedral feasibility problem and any linear programming
problem via standard homogenization procedures.

The projection and rescaling algorithm~\cite{PenaS16, PenaS19} is a recently developed  method to solve the feasibility problem 
\begin{equation}\label{primal}
\text{find}  \; x\in L\cap\R^n_{++},
\end{equation}
where $L\subseteq \R^n$ is a linear subspace.  
The gist of this algorithm is to combine two main steps, namely a {\em basic procedure}  and a {\em rescaling step}.  The basic procedure either finds a solution to~\eqref{primal} if this problem is {\em well-conditioned}, or determines a rescaling step that improves the {\em conditioning} of problem~\eqref{primal}.  In the latter case  a rescaling step is performed and the basic procedure is invoked again.  If $L\cap\R^n_{++} \ne \emptyset$ then this kind of iterative basic-procedure and rescaling-step scheme succeeds in finding a solution to~\eqref{primal} because the rescaling step eventually yields a sufficiently well-conditioned problem that the basic procedure can solve~\cite[Theorem 1]{PenaS16}.  The conditioning of  problem~\eqref{primal} is determined by a suitable measure of the {\em most interior} points in $L\cap \R^n_{++}$.

The projection and rescaling algorithm~\cite{PenaS16} was largely inspired by Chubanov's work~\cite{Chub12, Chub15} who developed this algorithm when $L$ is of the form $L = \{x\in\R^n: Ax  = 0\}$ for a matrix $A\in\R^{m\times n}$.  The projection and rescaling algorithm is in the same spirit as a number of articles based on the principle of enhancing a simple procedure with some sort of periodic reconditioning step~\cite{BellFV09, Betk04, DaduVZ16, DaduVZ17, DunaV06, Freund85, HobeR17, KitaT18, LourKMT16, PenaS13, PenaS19, Roos18}.

The projection and rescaling algorithm described in~\cite[Algorithm 1]{PenaS16} successfully terminates only when $L\cap \R^n_{++} \ne \emptyset$.  As it is described in~\cite[Algorithm 2]{PenaS16}, the algorithm has a straightforward extension that  terminates with a solution to \eqref{primal} or to its strict alternative problem
\begin{equation}\label{dual}
\text{find}  \; \hat x\in L^\perp\cap\R^n_{++},
\end{equation}
where $L^\perp := \{s: \ip{s}{x} = 0 \text{ for all } x\in L\}$,
provided that~\eqref{primal} or~\eqref{dual} is feasible.   In other words, the projection and rescaling algorithm~\cite[Algorithm 2]{PenaS16} solves the {\em full support problem} in the terminology of Dadush, Vegh, and Zambelli~\cite{DaduVZ17}.

The article~\cite{PenaS19} demonstrates the computational effectiveness of the projection and rescaling algorithm~\cite[Algorithm 2]{PenaS16}. For computational purposes, the implementation in~\cite{PenaS19} is a tweaked version of~\cite[Algorithm 2]{PenaS16} that applies to the two alternative problems
\begin{equation}\label{primal-dual}
\text{find}  \; x\in L\cap\R^n_{+}, \;\;\;\; \text{and } \; \;\;\;\;\text{find} \; \hat x\in L^\perp\cap\R^n_{+}
\end{equation}
without any prior full-support assumptions. 
The implementation in~\cite{PenaS19} aims to find {\em maximum support solutions} to the problems in~\eqref{primal-dual}
in the terminology of Dadush, Vegh, and Zambelli~\cite{DaduVZ17}. More precisely~\cite[Algorithm 1]{PenaS19} aims to find points $x\in L\cap\R^n_{+}$ and $\hat x \in L^\perp\cap\R^n_{+}$ such that the support sets
$\{j\in \{1,\dots,n\}: x_j > 0\}$ and $\{j\in \{1,\dots,n\}: \hat x_j > 0\}$
are maximal.  Notice that $x\in L\cap\R^n_{+}$ and $\hat x \in L^\perp\cap\R^n_{+}$ have maximum support if and only if  $x\in \text{relint}(L\cap\R^n_{+})$ and $\hat x \in \text{relint}(L^\perp\cap\R^n_{+})$ respectively.

The computational results in~\cite{PenaS19} demonstrate that~\cite[Algorithm 1]{PenaS19} indeed succeeds in solving the maximum support problem. 
More precisely,~\cite[Algorithm 1]{PenaS19} provides a full description of a publicly available MATLAB implementation and reports on extensive computational experiments that provide ample empirical evidence of the effectiveness of the projection and rescaling algorithm to solve the maximum support problem.
However, the correctness of~\cite[Algorithm 1]{PenaS19} was formally shown only in the special case when either~\eqref{primal} or~\eqref{dual} was feasible.  The main goal of this article is to give a formal proof of the correctness of a variant of~\cite[Algorithm 1]{PenaS19}  that finds maximum support solutions to the pair of problems in~\eqref{primal-dual} in full generality.  Our work is related to the maximum support algorithms described in~\cite[Section 4]{DaduVZ17}.  However there are several major differences.

First and foremost, the type of rescaling step in our algorithms is fundamentally different from that in~\cite{DaduVZ17}.
The approach in \cite{DaduVZ17} assumes that $L = \{x\in\R^n: Ax = 0\}$ and thus  $L^\perp = \{A\transp y: y\in \R^m\}$ for some $A\in \R^{m\times n}$.  The algorithms in \cite{DaduVZ17} perform a rescaling step of the form $A \mapsto MA$ for some non-singular $M\in \R^{m\times m}$.  In other words, the rescaling step reshapes the set $\{y\in\R^m:A\transp y \ge 0\} \subseteq \R^m$  while leaving $L$ and $L^\perp$ unchanged.  By contrast, our approach perform a rescaling of the form $L \mapsto DL$ for some diagonal $D\in\R^{n\times n}$ with positive diagonal entries.  In other words, our rescaling step reshapes $L\cap\R^n_+ \subseteq\R^n$ while leaving $\R^n_+$ unchanged.  
Second, our algorithm treats both problems $x\in L\cap \R^n_+$ and $\hat x\in L^\perp \cap \R^n_+$ jointly and in completely symmetric fashion.  Indeed, a key ingredient of our algorithm is the duality between these two problems.   Third, our algorithm is a natural extension of the projection and rescaling algorithm in~\cite{PenaS16} and inherits most of its simplicity.  Fourth, our results are stated entirely in the real model of computation.  Unlike~\cite{DaduVZ17}, we do not require the subspace $L$ to be of the form $\{x\in \R^n: Ax =0\}$ for some $A\in \mathbb{Z}^{m\times n}$.  Since our results apply to real data, they have no dependence at all on any bit-length encoding of the subspace $L$.  Instead, we show that the running time of our algorithm depends on suitable {\em condition measures} $\sigma(L)$ and $\sigma(L^\perp)$ of the relative interiors of  $L\cap\R^n_+$ and $L^\perp\cap\R^n_+$ respectively.  The condition measures $\sigma(L)$ and $\sigma(L^\perp)$ can be seen as an extension and refinement of the condition measure $\delta_\infty(L)$ for $L\cap\R^n_{++}$ proposed in~\cite{PenaS16}.  Fifth, although the analysis of our algorithm depends on the  condition measures $\sigma(L)$ and $\sigma(L^\perp)$, the algorithm does not require knowledge of them.  In short, our work  provides answers to the main open questions stated by Dadush, Vegh, and Zambelli~\cite[Section 5]{DaduVZ17}.

Our approach also yields, as a nice byproduct, a sharper analysis of the original projection and rescaling algorithm~\cite[Algorithm 1]{PenaS16} for the full-support case, that is, when $L\cap \R^n_{++}\ne \emptyset$ or $L^\perp\cap \R^n_{++}\ne \emptyset$. Furthermore, for the full-support case we compare the performance of our projection and rescaling algorithm and its condition-based analysis with the performance of the previous rescaling algorithms and their condition-based analyses described in~\cite{BellFV09,DaduVZ17,DunaV06,PenaS13}. 
The performance of the latter algorithms is stated in terms of a different condition measure $\rho(A)$ for the problems~\eqref{primal} and~\eqref{dual} that depends on a matrix $A\in \R^{m\times n}$ such that
$L = \{x\in \R^n: Ax=0\}$ or equivalently $L^\perp = \{A\transp y: y\in \R^m\}$.  
Thus our comparison concentrates on how the condition measures 
$\sigma(L)$ and $\sigma(L^\perp)$ used in this paper and the condition measure $\rho(A)$ used in~\cite{BellFV09,DaduVZ17,DunaV06,PenaS13} relate to each other.  
It is remarkable that these condition measures $\rho(A)$ and $\sigma(L),\sigma(L^\perp)$ can be related at all as they are associated to  geometric properties of objects in completely different spaces.  When $L^\perp\cap \R^n_{++}\ne \emptyset$ the quantity $\sigma(L^\perp)$ can be interpreted as a measure of how deep $L^\perp$ cuts inside $\R^n_+$ whereas $\rho(A)$ can be interpreted as the width of the cone $\{y\in\R^m: A\transp y \ge 0\}$.  Similar interpretations apply when $L\cap \R^n_{++}\ne \emptyset$.  Without any preconditioning assumptions, it is easy to construct counterexamples where either of the condition measures $\rho(A)$ or $\sigma(L^\perp)$ is arbitrarily larger than the other.  However, we show that after performing a natural and simple preconditioning of $A$, namely after normalizing the columns of $A$, the measures $\sigma(L)$ and $\sigma(L^\perp)$  are less conservative, and possibly far less so, than $\rho(A)$.  Consequently, our projection and rescaling algorithm applied to the full-support case has stronger condition-based convergence properties than the previous ones in~\cite{BellFV09,DaduVZ17,DunaV06,PenaS13}.

\medskip

The remaining sections of the paper are organized as follows. Section~\ref{sec.support} presents our main developments.  This section details our approach to finding  maximum support solutions to~\eqref{primal-dual}.  The approach hinges on three key ideas.  First, we propose an algorithm that finds a point in a set of the form $L\cap \R^n_+$ that may not necessarily have  maximum support (see Algorithm~\ref{algo.support}).  Second, by relying on the first algorithm and on a natural duality between the two problems in~\eqref{primal-dual}, we propose a second algorithm that finds maximum support solutions to~\eqref{primal-dual} (see Algorithm~\ref{algo.max.support}).  Third, the analyses and to some extent the design of our algorithms rely on suitable refinements of condition measures previously proposed and used in~\cite{PenaS16,Ye94} (see Proposition~\ref{prop:support} and Theorem~\ref{theor}).  Section~\ref{sec.full.support} shows how the new developments in Section~\ref{sec.support} automatically yield a sharpening of the analysis previously performed in~\cite{PenaS16} for the projection and rescaling algorithm in the full support case.  Section~\ref{sec.full.support} also shows that the condition measure $\rho(A)$ used in 
~\cite{BellFV09,DaduVZ17,DunaV06,PenaS13} for other rescaling algorithms
is more conservative, and possibly far more so, than $\sigma(L)$ and $\sigma(L^\perp)$ provided the columns of $A$ are normalized.
  Section~\ref{sec.bp} details an implementation of the basic procedure, which is a central building block of our projection and rescaling algorithm.  We limit our exposition to the most efficient known implementation of the basic procedure, namely a smooth perceptron scheme~\cite{PenaS16,PenaS19}.  The exposition in Section~\ref{sec.bp} highlights a simple and insightful but somewhat overlooked duality property that underlies the first-order implementations of the basic procedure  described in~\cite{PenaS16,PenaS19}.   Section~\ref{sec.variant} describes a variant of Algorithm~\ref{algo.support} that performs rescaling along multiple directions and Section~\ref{sec.conclusion} concludes the paper.

\section{Partial support and maximum support solutions.}
\label{sec.support}

We develop a two-step approach to finding maximum support solutions to~\eqref{primal-dual}.  The first step is Algorithm~\ref{algo.support} which  finds a point $x\in L \cap \R^n_+$ whose support may or may not be maximum.  By leveraging Algorithm~\ref{algo.support} and the duality between the two problems in~\eqref{primal-dual}, Algorithm~\ref{algo.max.support} finds maximum support solutions $x\in L \cap \R^n_+$ and 
$\hat x \in L^\perp \cap \R^n_+$. 
Our developments rely on several key constructions and pieces of notation detailed next.

For $x\in \R^n_+$, the {\em support} of $x$ is the set $\{j\in \{1,\dots,n\}: x_j > 0\}$.  Suppose $L\subseteq \R^n$ is a linear subspace. 
 Let $J(L)$ be the following {\em maximum support} index set \[
J(L) := \left\{j\in \{1,\dots,n\}: x_j > 0 \text{ for some } x\in L \cap\R^n_+\right\}.
\] 
In other words, $J(L)$ is the support of any point in the relative interior of $L\cap \R^n_{+}$ or equivalently the maximum support of all $x\in L\cap \R^n_+$.
It is evident that $J(L) = \{1,\dots,n\} \Leftrightarrow L\cap \R^n_{++} \ne \emptyset$ and $J(L) = \emptyset
\Leftrightarrow L\cap \R^n_{+} = \{0\}
 \Leftrightarrow L^\perp\cap \R^n_{++} \ne \emptyset$. 
As we formalize in the sequel, the difficulty of correctly identifying $J(L)$ and thus that of finding a maximum support solution to $x\in L\cap \R^n_+$ is determined by the following condition measure which is a variant of a condition measure proposed by Ye~\cite{Ye94}:
\[
\sigma(L) := \min_{j\in J(L)} \max \{x_j: x\in L \cap \R^n_+, \|x\|_\infty \le 1\}.
\]
The construction of $\sigma(L)$ implies  that $\sigma(L) \in (0,1]$ whenever $J(L) \ne \emptyset$.  For convenience we let $\sigma(L):= 1$ when $J(L) = \emptyset.$ This convention ensures that $\min\{\sigma(L),\sigma(L^\perp)\} = \sigma(L^\perp)$ when $J(L)=\emptyset.$

It is easy to see, via a standard separation argument or theorem of the alternative, that the maximum support index sets $J(L)$ and $J(L^\perp)$ partition $\{1,\dots,n\}$, that is, 
\[
J(L)\cap J(L^\perp) = \emptyset \text{ and } 
J(L)\cup J(L^\perp) = \{1,\dots,n\}.
\]
For $\sigma \in (0,1]$ let
\[
J_\sigma(L) = \{j\in\{1,\dots,n\}: x_j \ge \sigma \text{ for some } x\in L\cap \R^n_+, \|x\|_\infty \le 1\}. 
\] 
From the construction of $\sigma(\cdot), J(\cdot),$ and $J_\sigma(\cdot)$ it follows that for $\sigma \in (0,1]$
\[
J_\sigma(L) \subseteq J(L) \text{ and } J_\sigma(L) = J(L) \text{ if and only if } \sigma \le \sigma(L).
\]
For $J\subseteq \{1,\dots,n\}$ let $\R^J \subseteq \R^n$ denote the subspace
$
\{x\in \R^n: x_i= 0 \text{ for } i\not \in J\} 
$.  For nonempty $J\subseteq \{1,\dots,n\}$ let $\Delta(J)\subseteq \R^J$ denote the set $\{x\in \R^J: x \ge 0, \|x\|_1 = 1\}.$   Given a linear subspace $L\subseteq \R^n$ and $J\subseteq \{1,\dots,n\}$, let $P_{L\vert J}:\R^{J}\rightarrow L\cap \R^{J}$ denote the orthogonal projection from $\R^{J}$ onto $L\cap \R^{J}$.

Similar to the projection and rescaling algorithm in~\cite{PenaS16,PenaS19}, Algorithm~\ref{algo.support} consists primarily of two main steps, namely a {\em basic procedure} and a {\em rescaling phase.} 
Both of these steps are slight modifications of those in the original projection and rescaling algorithm~\cite[Algorithm 1]{PenaS19}. Our algorithm attempts to find $J(L)$ by gradually identifying and trimming indices presumed not to be in $J(L)$.  The trimming decision is based on whether the rescaling matrix has exceeded a certain predefined threshold determined by an educated guess $\sigma$ of $\sigma(L)$.  If the educated guess $\sigma$ is too large, the algorithm may mistakingly trim indices in $J(L)$.  Thus Algorithm~\ref{algo.support} is only guaranteed to find a solution $x\in L\cap \R^n_+$ with {\em partial support,} that is, a solution whose support may not be maximum. When $\sigma \leq \sigma(L)$, Algorithm~\ref{algo.support}  correctly identifies the maximum support set $J(L)$ (see  Proposition~\ref{prop:support}).  Although it is evidently desirable to choose the educated guess $\sigma$ deliberately small, this comes at a cost as Proposition~\ref{prop:support} shows.

Algorithm~\ref{algo.support} works as follows.  Start with the support set $J = \{1,\ldots,n\}$, rescaling matrix $D = I$, and an educated guess $\sigma > 0$ of $\sigma(L)$.  At this initial stage no indices are trimmed and the problem is not rescaled.  At each main iteration the basic procedure either finds $u \in \Delta(J)$ such that 
$(P_{DL\vert J}u)_J > 0$  or finds $z \in \Delta(J)$ such  $\|(P_{DL \vert J}z)^+\|_1 \le \frac{1}{2} \|z\|_\infty$.  When $(P_{DL\vert J}u)_J > 0$ the algorithm outputs the point $x:=D^{-1}P_{DL \vert J} u \in L\cap \R^n_+$ with support $J$.  When $\|(P_{DL \vert J}z)^+\|_1 \le \frac{1}{2} \|z\|_\infty$ the algorithm updates the rescaling matrix $D$ to improve the conditioning of $DL\cap \R^n_+$ and trims $J$ if some entries in $D$ exceed the threshold $1/\sigma$.  The main difference between Algorithm~\ref{algo.support} and the projection and rescaling algorithm in~\cite{PenaS16} is the support set $J$ and its dynamic adjustment after each rescaling step.

Section~\ref{sec.bp} below describes a possible first-order implementation of the basic procedure that terminates in at most $\Oh(n^{1.5})$ low-cost iterations.  Other first-order implementations are discussed in~\cite{PenaS16,PenaS19} all of which terminate in at most $\Oh(n^3)$ low-cost iterations.

The following technical lemma formalizes how the rescaling step improves the conditioning of $L\cap \R^n_+$. To that end, we will rely on one additional piece of notation.  Suppose $L\subseteq \R^n$ is a linear subspace and $i\in \{1,\dots,n\}$.  Let
\[
\sigma_i(L):=\max\{x_i: x\in L\cap \R^n_+, \|x\|_\infty \le 1\}.
\]
Observe that $J(L)=\{i: \sigma_i(L)>0\}$ and $\sigma(L) = \min_{i\in J(L)} \sigma_i(L).$  
Thus the vector $\matr{\sigma_1(L) & \cdots & \sigma_n(L)}$  encodes both $J(L)$ and  $\sigma(L)$ but is more informative about the conditioning of $L\cap\R^n_+$ than $J(L)$ and $\sigma(L)$.

 We will use the following common notational convention: for $i\in \{1,\dots,n\}$ let $e_i\in \R^n$ denote the vector whose $i$-th entry is one and all other entries are zero.

\begin{lemma}\label{lemma} Let $L\subseteq \R^n$ be a linear subspace and $P:\R^n\rightarrow L$ be the orthogonal projection onto $L$.  Suppose $z\in\R^n_+ \setminus\{0\}$ is such that $\|(Pz)^+\|_1\le \frac{1}{2} \|z\|_\infty = \frac{1}{2}z_i$ for some $i \in \{1,\dots,n\}$. Then 
\begin{equation}\label{eq.cut}
x\in L \cap \R^n_+ \Rightarrow x_i \le \frac{\|x\|_\infty}{2}.
\end{equation}
In particular, for $D = I+e_ie_i\transp$ the rescaled subspace $DL\subseteq \R^n$ satisfies
\[
\sigma_i(DL) = 2\sigma_i(L) \text{ and } \sigma_j(DL) = \sigma_j(L) \text{ for } j\ne i.
\]
\end{lemma}
\begin{proof}
If $x\in L \cap \R^n_+$ then
\[
0 \le x_iz_i \le \ip{x}{z} = \ip{Px}{z} = \ip{x}{Pz} \le \ip{x}{(Pz)^+} \le \|x\|_\infty \cdot \|(Pz)^+\|_1 \le \frac{\|x\|_\infty\cdot z_i}{2}.
\] 
Thus~\eqref{eq.cut} follows.

Next,~\eqref{eq.cut} implies that $
\{x\in L\cap\R^n_+: \|x\|_\infty \le 1\}=
\{x\in L\cap\R^n_+: \|Dx\|_\infty \le 1\}.$  Therefore for $j=1,\dots,n$ 
\begin{align*}
\sigma_j(DL) &= \max\{(Dx)_j: x\in L\cap\R^n_+, \|Dx\|_\infty \le 1\} \\
& = \max\{(Dx)_j: x\in L\cap\R^n_+, \|x\|_\infty \le 1\}\\
& = D_{jj} \max\{x_j: x\in L\cap\R^n_+, \|x\|_\infty \le 1\}\\
&= D_{jj} \sigma_j(L).
\end{align*}
That is, $\sigma_i(DL) = 2\sigma_i(L)$ and $\sigma_j(DL) = \sigma_j(L)$ for $j\ne i$.
\qed
\end{proof}

{\centering\begin{minipage}{\linewidth}
\begin{algorithm}[H]
  \caption{Partial support} \label{algo.support}
  \begin{algorithmic}[1]
    \State  ({\bf Initialization}) 
    \Statex Let $D := I$, $ J := \{1,\dots,n\}$, and $\sigma\in (0,1)$ be an educated  guess of $\sigma(L)$.
       \State Let $P := P_{DL\vert J}$
	\State ({\bf Basic Procedure})
	\Statex 	\quad Find  either 
	$u\in \Delta(J)$ such that $(Pu)_J > 0$ or 
	\Statex \quad $z \in \Delta(J)$ such that $\|(Pz)^+\|_1 \le \frac{1}{2} \|z\|_\infty$.
	\State   {\bf If} $(Pu)_J > 0$ {\bf then} HALT and output  $x = D^{-1}Pu$ and $J$

\State {\bf Else (Rescale $L$ \& Trim $J$) }
\Statex \quad let $i:= \argmax_j z_j$ and $D:=(I+e_ie_i\transp)D$
\Statex \quad {\bf if} $D_{ii} > 1/\sigma$ {\bf then} let $J = J \setminus \{i\}$
\Statex \quad {\bf if } $J = \emptyset$ {\bf then} HALT and output  $x = 0$ and $J=\emptyset$
	\Statex \quad Go back to step 2
	\end{algorithmic}
\end{algorithm}
\end{minipage}}

\begin{proposition}\label{prop:support} Suppose $\sigma \in (0,1)$.  Then Algorithm~\ref{algo.support} finds $x\in L\cap \R^n_+$ such that $x_J > 0$ for some $J_\sigma(L)\subseteq J \subseteq J(L) $ in at most 
\begin{equation}\label{eq.bound.partial}
\sum_{i\in J_{\sigma}(L)} \lceil\log_2(1/\sigma_i(L))\rceil + \left(n-\vert J_\sigma(L)\vert\right) \lceil\log_2(1/\sigma)\rceil \le  n \lceil\log_2(1/\sigma)\rceil\end{equation}
rescaling steps. Furthermore $J = J(L)$ if $\sigma \le \sigma(L)$.
\end{proposition}

\begin{proof}  First, observe that the algorithm must eventually halt since each entry $D_{ii}$ can be rescaled only up to $\lceil \log_2(1/\sigma)\rceil$ times before $i$ is trimmed from $J$. Lemma~\ref{lemma} implies that throughout the algorithm 
\[
\sigma_i(DL) = D_{ii} \cdot \sigma_i(L)
\]
for $i=1,\dots,n$. In particular, $D_{ii} \le 1/\sigma_i(L)$ for $i\in J(L)$ because $\sigma_i(DL) \le 1$.  On the other hand, the trimming operation implies that $\log_2(D_{ii}) \le \lceil \log_2(1/\sigma)\rceil$ for $i=1,\dots,n.$  
Hence when the algorithm halts, the total number of rescaling steps that the algorithm has performed is
\begin{align*}
\sum_{i=1}^n \log_2(D_{ii}) &= 
\sum_{i\in J_\sigma(L)} \log_2(D_{ii}) + 
\sum_{i\not\in J_\sigma(L)} \log_2(D_{ii}) \\
&\le \sum_{i\in J_{\sigma}(L)} \lceil\log_2(1/\sigma_i(L))\rceil + \left(n-\vert J_\sigma(L)\vert\right) \lceil\log_2(1/\sigma)\rceil\\
&\le n\lceil\log_2(1/\sigma)\rceil.
\end{align*}
  If the algorithm terminates with $J= \emptyset$ then the output solution $x = 0 \in L \cap \R^n_+$  vacuously satisfies $x_J > 0$.  Otherwise the algorithm outputs $x = D^{-1}Pu$ for $Pu \in DL\cap \R^J$ and $(Pu)_J > 0$.  Therefore the output solution $x$ satisfies $x = D^{-1}P u \in L \cap \R^n_+$ and $x_J > 0$.
  
We next show that upon termination $J_\sigma(L)\subseteq J \subseteq J(L)$.  Indeed, since $D_{ii} \le 1/\sigma_i(L)$ for $i\in J(L)$, we have $D_{ii} \le 1/\sigma_i(L) \le 1/\sigma$ for $i\in J_\sigma(L)$ and thus the algorithm never trims any indices in $J_\sigma(L)$.  Thus $J_\sigma(L)\subseteq J$ upon termination.
On the other hand, upon termination $J\subseteq J(L)$ since the algorithm outputs some $x\in L\cap \R^n_+$ with $x_J > 0$.
Finally, if $\sigma \leq\sigma(L)$ then $J(L) =J_{\sigma}(L)  \subseteq J \subseteq J(L)$ and so 
$J =J(L)$.
\qed
\end{proof}

\bigskip

Algorithm~\ref{algo.support} has the following natural variant that performs a rescaling along multiple directions:  in Step 5 let $e:=(z/\|(Pz)^+)\|_1-1)$ and use $I+\text{Diag}(e)$ instead of $(I+e_ie_i\transp)$.  We discuss this variant in detail in Section~\ref{sec.variant} where we also show that it has the same convergence properties of Algorithm~\ref{algo.support}.

\bigskip

Algorithm~\ref{algo.support} provides a {\em conceptual} method to find maximum support solution for the feasibility problem  $x \in L \cap \R^n_+$, namely: call Algorithm~\ref{algo.support} with $\sigma \le \sigma(L)$.  Furthermore, Proposition~\ref{prop:support} suggests that the ideal choice is $\sigma = \sigma(L)$.  However, this conceptual method is doomed as $\sigma(L)$ is generally unknown and computing it appears to be as hard as or harder than finding a maximum support solution in $L\cap \R^n_+$.  
We circumvent this challenge via a simple but clever procedure to check if a solution found by Algorithm~\ref{algo.support} has maximum support.  To do so, we rely on the fact that the sets $J(L)$ and  $J(L^\perp)$ partition $\{1,\dots,n\}$.  This is a key duality connection between the problems \eqref{primal-dual}.  More precisely, we apply Algorithm~\ref{algo.support} to both $L\cap \R^n_+$ and $L^\perp\cap\R^n_+$ simultaneously and rely on the observation formalized in Corollary~\ref{corol} 
to determine when the algorithm has found maximum support solutions.

Corollary~\ref{corol} and Algorithm~\ref{algo.max.support} rely on the following notation.
Given a linear subspace $L\subseteq \R^n$ and $\sigma \in (0,1)$, let $(x, J) := \J (L,\sigma)$ denote the output of
Algorithm~\ref{algo.support} when called with input $(L,\sigma)$. The following result readily follows.

\begin{corollary}~\label{corol} Let $L\subseteq \R^n$ be a linear subspace and $\sigma > 0$.  If $(x,J) = \J(L,\sigma)$ and $(\hat x,\hat J) = \J(L^\perp,\sigma)$ satisfy $J \cup \hat J = \{1,\dots,n\}$ then $J = J(L), \hat J = J(L^\perp)$, and $x,\hat x$ are  maximum support points in $L\cap \R^n_+$ and $L^\perp\cap \R^n_+$ respectively.
\end{corollary}

\begin{proof}
Proposition~\ref{prop:support} implies that $x\in L\cap \R^n_+, \; \hat x\in L^\perp\cap \R^n_+$ with $x_J > 0, \; \hat x_{\hat J} > 0$ and  $J_\sigma(L) \subseteq J \subseteq J(L),\; J_\sigma(L^\perp) \subseteq \hat J \subseteq J(L^\perp)$.  Since the index sets $J(L)$ and $J(L^\perp)$ partition $\{1,\dots,n\}$, the identity $J \cup \hat J= \{1,\dots,n\}$ can only occur when $J=J(L)$ and $\hat J = J(L^\perp).$ 
\qed
\end{proof}

Corollary \ref{corol} naturally suggests the following iterative strategy
to find maximum support solutions to~\eqref{primal-dual}. 
Start with the ad-hoc initial guess $\sigma=1/2$ of $\min\{\sigma(L),\sigma(L^\perp)\}$ and let  $(x,J):=\J(L,\sigma)$ and $(\hat x,\hat J):=\J(L^\perp,\sigma)$.  If $J \cup \hat J= \{1,\dots,n\}$ then Corollary~\ref{corol}
implies that we found maximum support solutions to~\eqref{primal-dual}.  Otherwise, reduce $\sigma$ and repeat.  
Algorithm~\ref{algo.max.support} formally describes the  above strategy.
Theorem \ref{theor} shows that if $\sigma$ is reduced by squaring it each time, then this iterative procedure succeeds after at most 
$\Oh(\log_2\log_2(1/\min\{\sigma(L), \sigma(L^\perp)\}))$ guessing rounds and 
$\Oh(n\log_2(1/\min\{\sigma(L), \sigma(L^\perp)\}))$ rescaling steps.

{\centering\begin{minipage}{\linewidth}
\begin{algorithm}[H]
  \caption{Maximum support     \label{algo.max.support}}
  \begin{algorithmic}[1]
    \State  Take $\sigma:=1/2$ as initial  guess of $\min\{\sigma(L),\sigma(L^\perp)\}$.
       \State Let $(x,J):=\mathcal J(L,\sigma)$ and $(\hat x,\hat J):=\mathcal J(L^\perp,\sigma)$
	\State {\bf If} $J\cup \hat J = \{1,\dots,n\}$ {\bf then} HALT
\State {\bf Else} (scale down $\sigma$) 
\Statex \quad let $\sigma := \sigma^2$
	\Statex \quad Go back to step 2.
	\end{algorithmic}
\end{algorithm}
\end{minipage}}

\begin{theorem}\label{theor} Upon termination Algorithm~\ref{algo.max.support} 
correctly identifies $J=J(L), \; \hat J = J(L^\perp)$ and finds $x\in L\cap \R^n_+, \;  \hat x\in L^\perp\cap \R^n_+$ with $x_J > 0, \hat x_{\hat J}>0$.  The algorithm terminates after  at most 
\begin{equation} \label{eq.main.iter}
 k = \left\lceil\log_2\left(\log_2\left(1/\min\left\{\sigma(L), \sigma(L^\perp)\right\}\right)\right)\right\rceil+1
 \end{equation} 
 main iterations and the total number of rescaling steps performed by Algorithm~\ref{algo.max.support} is bounded above by  
\begin{equation}\label{eq.total.iter}
4n 
\left\lceil\log_2 (1/\min\{\sigma(L), \sigma(L^\perp)\})\right\rceil.
\end{equation}
\end{theorem}
\begin{proof} Proposition~\ref{prop:support}  implies that  $J = J(L)$ and $\hat J = J(L^\perp)$ and thus 
$J\cup\hat J = \{1,\dots,n\}$ when $\sigma \leq \min\left\{\sigma(L), \sigma(L^\perp)\right\}$ or possibly sooner.  Corollary~\ref{corol} hence implies that the solutions $x$ and $\hat x$ returned by Algorithm~\ref{algo.max.support} are maximum support solutions.  Since $\sigma$ is squared at every iteration starting at $1/2$, the number $k$ of main iterations performed by Algorithm~\ref{algo.max.support} is at most the smallest $k$ that satisfies $$\frac{1}{2^{2^{k-1}}} \le 
\min\left\{\sigma(L), \sigma(L^\perp)\right\}.$$ Thus~\eqref{eq.main.iter} follows.

At each main iteration $i=1,\dots,k$ Algorithm~\ref{algo.max.support} calls Algorithm~\ref{algo.support} twice with input pairs $(L,1/2^{2^{i-1}})$ and $(L^\perp,1/2^{2^{i-1}})$. Proposition~\ref{prop:support} implies that each of these calls terminates after at most $ n\lceil\log_2(2^{2^{i-1}})\rceil = n \cdot 2^{i-1}$ rescaling steps.  Hence the total number of rescaling steps performed by Algorithm~\ref{algo.max.support} is bounded above by
\[
\sum_{i=1}^{k}  2n  \cdot 2^{i-1}
= 2n\cdot(2^k-1) \le 4n\lceil\log_2(1/\min\{\sigma(L), \sigma(L^\perp)\})\rceil.
\]
\qed
\end{proof}

It is worth noticing that Algorithm~\ref{algo.max.support} circumvents the main limitation of Algorithm~\ref{algo.support} with barely any overhead: if we knew both $\sigma(L)$ and $\sigma(L^\perp)$ then we could find maximum support solutions via  Algorithm~\ref{algo.support} by letting $(x,J) = \J(L,\sigma(L))$ and $(\hat x, \hat J)= \J(L^\perp,\sigma(L^\perp))$ with a total number of rescaling steps bounded above by
\[
 n \lceil\log_2(1/\sigma(L))\rceil + n \lceil\log_2(1/\sigma(L^\perp))\rceil. 
\]
Therefore, the bound~\eqref{eq.total.iter} is at most four times larger than the {\em ideal} bound achievable if we had full knowledge of both $\sigma(L)$ and $\sigma(L^\perp)$.

\medskip

The dependence on $\min\{\sigma(L),\sigma(L^\perp)\}$ in Theorem~\ref{theor} naturally raises the question: are the magnitudes of $\sigma(L)$ and $\sigma(L^\perp)$ related to each other?  As our next generic construction shows, the answer to this question is no.    Indeed, either $\sigma(L)$ or $\sigma(L^\perp)$ may be arbitrarily larger than the other.  

\begin{proposition}\label{prop.sigmas} Suppose $L=L_1\times L_2\subseteq \R^{n_1+n_2}$ where $L_1\subseteq \R^{n_1}$ and $L_2\subseteq \R^{n_2}$ are such that
\[
L_1\cap \R^{n_1}_{++}\ne\emptyset \text{ and } L_2^\perp\cap \R^{n_2}_{++}\ne\emptyset.
\]
Then $\sigma(L) = \sigma(L_1)$ and $\sigma(L^\perp) = \sigma(L_2^\perp).$  In particular, either $\sigma(L)$ or $\sigma(L^\perp)$ can be arbitrarily larger than the other one. 
\end{proposition}
\begin{proof}
A straightforward verification shows that
\[
L\cap \R^{n_1+n_2}_+ = \left(L_1\cap \R^{n_1}_{+} \right)\times \{0_{n_2}\} \;\text{ and }\; 
L^\perp\cap \R^n_+ = \{0_{n_1}\} \times \left(L_2^\perp \cap \R^{n_2}_{+}\right).
\]
In particular,
\[
\left(L_1\cap \R^{n_1}_{++} \right) \times \{0_{n_2}\} \subseteq L\cap \R^n_+,\; \{0_{n_1}\} \times \left(L_2^\perp \cap \R^{n_2}_{++}\right)  \subseteq L^\perp\cap \R^n_+.\]
It thus follows that  $J(L) = \{1,\dots,n_1\},\; J(L^\perp) = \{n_1+1,\dots,n_1+n_2\}$ and also $\sigma(L) = \sigma(L_1)$ and $\sigma(L^\perp) = \sigma(L_2^\perp)$.

To finish, observe that either $\sigma(L_1)$ or $\sigma(L_2^\perp)$ can be arbitrarily larger than the other since the spaces $L_1\subseteq \R^{n_1}$ and $L_2\subseteq \R^{n_2}$ have no dependence on each other. 
\qed
\end{proof}

Proposition~\ref{prop.sigmas} also suggests why the dependence on $\min\{\sigma(L),\sigma(L^\perp)\}$ appears to be inevitable in Theorem~\ref{theor}.  Suppose $\sigma(L^\perp) < \sigma < \sigma(L)$ which is perfectly possible.  Then Step 2 of Algorithm~\ref{algo.max.support} finds $(x,J) = \mathcal J(L,\sigma)$ and 
$(\hat x,\hat J) = \mathcal J(L^\perp,\sigma)$ with $J = J(L)$ but not necessarily $\hat J = J(L^\perp)$.  However, since there is no a priori relationship between $\sigma(L)$ and $\sigma(L^\perp)$, as long as $J \cup \hat J \ne \{1,\dots,n\}$ the algorithm only knows that either $J \ne J(L)$ or $\hat J \ne J(L^\perp).$  The algorithm thus needs to reduce $\sigma$ and perform Step 2 again until $J \cup \hat J = \{1,\dots,n\}$ holds.

\medskip

We conclude this section with a bound on the total number of arithmetic operations required by Algorithm~\ref{algo.max.support}.  We only give a loose bound since the interesting complexity bounds are already stated in Proposition~\ref{prop:support} and Theorem~\ref{theor}.  The bounds below can be sharpened via a more detailed and lengthly but not necessarily more insightful accounting of arithmetic operations.  In particular, to keep our exposition simple, we state the bounds only in terms of the dimension $n$ of the ambient space and ignore the potentially much lower dimension of $L$ or $L^\perp$.

As we detail in Section~\ref{sec.bp}, the smooth perceptron scheme for the basic procedure is guaranteed to terminate in $\Oh(n^{1.5})$ iterations.  The most costly operation in each iteration of the smooth perceptron is a matrix-vector multiplication involving the projection matrix $P_{DL\vert J}$, that is, $\Oh(n^2)$ arithmetic operations. Thus the number of arithmetic operations required by each call to the basic procedure is bounded above by
\[
\Oh(n^{3.5}).
\]
The number of arithmetic operations required by the rescaling and trimming step (even if we computed the projection matrix from scratch) is dominated by $\Oh(n^{3.5})$.  Therefore the total number of arithmetic operations required by Algorithm~\ref{algo.max.support} is bounded above by
\[
\Oh\left(n^{4.5} \cdot \lceil\log_2(1/\min\{\sigma(L),\sigma(L^\perp)\})\rceil\right).
\]

\section{Full support solutions redux.}
\label{sec.full.support}
We next revisit the projection and rescaling algorithm in~\cite{PenaS16} for the full support problem. We also compare its condition-based performance to that of the methods in~\cite{BellFV09,DaduVZ17,DunaV06,PenaS13}. 
Algorithm~\ref{algo.full.support} describes the projection and rescaling algorithm in~\cite{PenaS16}.  
Observe that Algorithm~\ref{algo.full.support} is the same as Algorithm~\ref{algo.support} without trimming.  Indeed, when $L\cap \R^n_{++} \ne \emptyset$ or equivalently $J(L) = \{1,\ldots, n\}$,  Algorithm~\ref{algo.full.support} does exactly the same as Algorithm~\ref{algo.support}
provided $\sigma \le \sigma(L)$.

{\centering\begin{minipage}{\linewidth}
\begin{algorithm}[H]
  \caption{Full support} \label{algo.full.support}
  \begin{algorithmic}[1]
    \State  ({\bf Initialization}) 
    \Statex Let $D := I$
       \State Let $P := P_{DL}$
	\State ({\bf Basic Procedure})
	\Statex \quad	Find  either $u\in \Delta_{n-1}$ such that $Pu > 0$ or 
	\Statex \quad $z \in \Delta_{n-1}$ such that $\|(Pz)^+\|_1 \le \frac{1}{2} \|z\|_\infty$
	\State    {\bf If} $Pu > 0$ {\bf then} HALT and output  $x = D^{-1}Pu$

\State {\bf Else (Rescale $L$)}
\Statex \quad let $i:= \argmax_j z_j$ and $D:=(I+e_ie_i\transp)D$

	\Statex \quad Go back to step 2
	\end{algorithmic}
\end{algorithm}
\end{minipage}}

\bigskip

Theorem 1 in~\cite{PenaS16} shows that when $L\cap\R^n_{++} \neq \emptyset$ Algorithm~\ref{algo.full.support} finds a full support solution in $L\cap\R^n_{++}$ in at most $\log_2(1/\delta_\infty(L))$ rescaling iterations where $\delta_\infty(L)$ is the following measure of the most interior solution to $L\cap\R^n_{++}$:
 $$\delta_\infty(L) = 
\max\left\{\prod_{j=1}^n x_j: x\in L\cap \R^n_{++}, \|x\|_\infty \le 1\right\}.$$ 
More precisely, the statement in~\cite[Theorem 1]{PenaS16} is actually stated in terms of a variant $\delta(L)$ of $\delta_\infty(L)$ that uses the normalization $\|x\|_2^2 \le n  $ instead of $\|x\|_\infty \le 1$.  However, as sketched in~\cite[page 93]{PenaS16}, the above statement in terms of $\delta_\infty(L)$ follows via a straightforward modification of~\cite[Theorem 1]{PenaS16}.

It is easy to see that $\prod_{j=1}^n \sigma_j(L) \ge \delta_\infty(L).$
Proposition~\ref{prop:improvedBound.full.support} below shows that the iteration bound $\log_2(1/\delta_\infty(L))$ can be sharpened to 
\[
\sum_{j=1}^n \log_2(1/\sigma_j(L)) = 
\log_2 \left(\prod_{j=1}^n 1/\sigma_j(L)\right)
\]
modulo some rounding.

\begin{proposition}\label{prop:improvedBound.full.support} If $L\cap \R^n_{++} \ne \emptyset$ then Algorithm~\ref{algo.full.support} finds $x\in L \cap \R^n_{++}$ in at most 
\begin{equation}\label{eq.bound.full}
\sum_{j=1}^n \lceil\log_2(1/\sigma_j(L))\rceil \le n\lceil\log_2(1/\sigma(L))\rceil
\end{equation}
rescaling steps.
\end{proposition}
\begin{proof}
This readily follows from Proposition~\ref{prop:support} since Algorithm~\ref{algo.full.support} is identical to  
Algorithm~\ref{algo.support} applied to $\sigma  \le  \sigma(L)$.  Indeed, for this choice of $\sigma$ we have $J_\sigma(L) = J(L) = \{1,\dots,n\}$ and thus Algorithm~\ref{algo.support} does not trim any indices and the first expression in~\eqref{eq.bound.partial} yields precisely the first expression in~\eqref{eq.bound.full}.
\qed
\end{proof}

It is evident that Algorithm~\ref{algo.full.support} terminates only when $L\cap\R^n_{++}\ne\emptyset$.  Proceeding exactly as in~\cite[Algorithm 2]{PenaS16}, we can apply Algorithm~\ref{algo.full.support} in parallel so that it terminates with either a solution in $L\cap\R^n_{++}$ or in $L^\perp\cap\R^n_{++}$ as long as one of them is nonempty in a number of rescaling iterations either bounded above by~\eqref{eq.bound.full} when $L\cap\R^n_{++}\ne \emptyset$ or bounded above by
\[
\sum_{j=1}^n \lceil\log_2(1/\sigma_j(L^\perp))\rceil \le n\lceil\log_2(1/\sigma(L^\perp))\rceil
\]
when $L^\perp\cap\R^n_{++}\ne \emptyset$.

\medskip

It is natural to ask how our projection and rescaling algorithm and its condition-based analysis compares with other rescaling algorithms and their condition-based analyses such as those described in~\cite{BellFV09,DaduVZ17,DunaV06,PenaS13}.  The condition-based analyses in all of these previous articles applies only to the full support case and are stated in terms of a different condition measure $\vert \rho(A)\vert$ where $A\in \R^{m\times n}$ is such that 
$L = \{x\in \R^n: Ax = 0\}$ or equivalently 
$L^\perp = \{A\transp y: y\in \R^m\}$. 
The dependence on $\rho(A)$ of all of the algorithms in~\cite{BellFV09,DaduVZ17,DunaV06,PenaS13} is due to the type of rescaling used: in contrast to the diagonal rescaling step $L \mapsto DL$ that lies at the core of our algorithms,  all of the methods in~\cite{BellFV09,DaduVZ17,DunaV06,PenaS13} rely on rescaling steps of the form $A \mapsto MA$ for some non-singular $M \in \R^{m\times m}$.  A rescaling step of the form $A\mapsto MA$ can be interpreted as an attempt to transform the set $\{y\in \R^m : A\transp y \ge 0\} \subseteq \R^m$ to one with a more favorable shape.  By contrast, a rescaling step of the form $L\mapsto DL$, as the one underlying our approach, can be interpreted as an attempt to transform the intersection $L\cap \R^n_+\subseteq \R^n$ to one with a more favorable shape.

Since the condition measures $\min\{\sigma(L),\sigma(L^\perp)\}$ and $\vert \rho(A) \vert$  are associated to objects in different spaces, it is not obvious that they are comparable to each other.   Indeed, without any preconditioning assumptions on $A$ either $\min\{\sigma(L),\sigma(L^\perp)\}$ or $\vert \rho(A) \vert$ can be arbitrarily larger than the other one.   Nonetheless,
Proposition~\ref{prop.compare}  and Proposition~\ref{prop.compare.reverse} below show that $\min\{\sigma(L),\sigma(L^\perp)\}$ and $\vert\rho(A)\vert$ can be bounded in terms of each other provided $A$ is suitably preconditioned. We should note that the bounds in Proposition~\ref{prop.compare}  and Proposition~\ref{prop.compare.reverse} are in the same spirit and similar to some results in~\cite{EpelF00,PenaRS14}.

Proposition~\ref{prop.compare} shows that  after a simple preconditioning step that leaves $\vert \rho(A)\vert$ unchanged the condition measure $\vert\rho(A)\vert$ is more conservative, and possibly far more so, than $\min\{\sigma(L),\sigma(L^\perp)\}$.  In particular, any algorithm whose condition-based analysis is stated in terms of $\min\{\sigma(L),\sigma(L^\perp)\}$ is automatically stronger, and possibly vastly so, than any other algorithm whose condition-based analysis is stated in terms of $\vert\rho(A)\vert$ as far as the dependence on the condition measure goes.  Said differently, the projection and rescaling algorithm for~\eqref{primal-dual}  described in this paper as well as its predecessor~\cite[Algorithm 2]{PenaS16} applied to the full-support case have stronger condition-based convergence properties than those in~\cite{BellFV09,DaduVZ17,DunaV06,PenaS13}.

\medskip

The condition measure $\rho(A)$ is defined as follows.  Suppose  $A \in \R^{m\times n}$ is a full row-rank matrix whose columns  are all non-zero.  That is,
\[
A = \matr{a_1 & \cdots & a_n} \in \R^{m\times n} \; \text{ with } \; a_i\ne 0 \text{ for } i=1,\dots,n.
\] 
The condition measure $\rho(A)$ is defined as follows
\[
\rho(A):=\max_{\|y\|_2 = 1} \min_{i=1,\dots,n} \frac{\ip{a_i}{y}}{\|a_i\|_2}.
\]
The condition measure $\rho(A)$ has an interesting history in optimization as discussed in~\cite{BurgC13,CheuC01,DunaV06,EpelF00,Goff80,SoheP12}.  Among other features, it has the following nice geometric interpretation.
When $\rho(A)>0$, the quantity $\rho(A)$ can be interpreted as a measure of {\em thickness} of the cone $K := \{y \in \R^m: A\transp y \ge 0\}\subseteq \R^m$.  Indeed, in this case  $\rho(A)$ is the radius of the largest ball centered at a point of Euclidean norm one and contained in $K$.  Furthermore, $\rho(A) > 0$ if and only if $0 \not \in \{Ax: x\in \Delta_{n-1}\}$ and $\rho(A)<0$ if and only if $0\in\text{int}(\{Ax: x\in \Delta_{n-1}\})$, where $\Delta_{n-1}:=\{x\in\R^n_+:\|x\|_1 = 1\}.$  We note  that the set $\{Ax: x\in \Delta_{n-1}\}$   is precisely the convex hull of the columns of $A$.   
Regardless of the sign of $\rho(A)$, its absolute value $\vert \rho(A) \vert$ is precisely the distance from $0$ to the boundary of $\{Ax: x\in \Delta_{n-1}\}$ provided all columns of $A$ have Euclidean norm equal to one, that is, $\|a_i\|_2=1,\; i=1\dots,n$.  

\medskip

Let $L = \{x\in \R^n: Ax = 0\}$ or equivalently
$L^\perp = \{A\transp y: y\in \R^m\}$.  It is easy to see that $\rho(A) > 0 \Leftrightarrow L^\perp \cap \R^n_{++} \ne \emptyset$ and $\rho(A) < 0 \Leftrightarrow L \cap \R^n_{++} \ne \emptyset$.  
Proposition~\ref{prop.compare} below refines these equivalences in terms of the condition measure  $\min\{\sigma(L), \sigma(L^\perp)\}.$  Note that our construction of $\sigma(\cdot)$ implies that $\sigma(L) = \min\{\sigma(L),\sigma(L^\perp)\}$ when $L\cap\R^n_{++} \ne \emptyset$ and $\sigma(L^\perp) = \min\{\sigma(L),\sigma(L^\perp)\}$ when $L^\perp\cap\R^n_{++} \ne \emptyset$.
Note also that $\rho(A)$ is invariant under positive scaling of the columns of $A$.  In particular, if the columns of $A\in \R^{m\times n}$ are all non-zero then $\rho(A) = \rho(\hat A)$ where $\hat A\in \R^{m\times n}$ is obtained by normalizing (positive scaling) the columns of $A$ so that all the columns of $\hat A$ have Euclidean norm equal to one.

\begin{proposition}\label{prop.compare}
Suppose  $A \in \R^{m\times n}$ is a full row-rank matrix whose columns have Euclidean norm equal to one. Let $L = \{x\in \R^n: Ax = 0\}$ or equivalently
$L^\perp = \{A\transp y: y\in \R^m\}$.  

\begin{description}
\item[(a)] If $\rho(A) > 0$ then $\rho(A) \le \sigma(L^\perp)$.  Furthermore, $\sigma(L^\perp)$ can be arbitrarily larger than $\rho(A)$.
\item[(b)] 
If $\rho(A) < 0$ then $\vert \rho(A)\vert \le \sigma(L)$.  Furthermore, $\sigma(L)$ can be arbitrarily larger than $\vert\rho(A)\vert$.
\end{description}
\end{proposition}
\begin{proof}
\begin{description}
\item[(a)] Let $\bar y\in \R^m$ be such that 
$\|\bar y\|_2 =1$ and 
$\rho(A) = \min_{i=1,\dots,n}\ip{a_i}{\bar y} > 0.$
  Then $\bar x := A\transp \bar y \in L^\perp$ and for each $i=1,\dots,n$ we have 
$\bar x_i = \ip{a_i}{\bar y}\ge \rho(A) > 0$ and $\bar x_i \le \|a_i\|_2\cdot \|\bar y\|_2 \le 1.$  In other words, $\bar x \in L\cap \R^{n}_{++}, \|\bar x\|_\infty \le 1,$ and $\bar x_i \ge \rho(A)$ for each $i=1,\dots,n$.  Thus $\sigma(L^\perp) \ge \rho(A).$ The following example shows that  $\sigma(L^\perp)$ can be arbitrarily larger than $\rho(A).$ Let 
\[
A = \frac{1}{\sqrt{1+\epsilon^2}}\matr{1 & 1 & -1 &-1\\ \epsilon&\epsilon&\epsilon&\epsilon}
\]
where $0<\epsilon< 1$.  It is easy to see that
$\rho(A) = \epsilon/\sqrt{1+\epsilon^2}$ and $\sigma(L^\perp) = 1$. 

\item[(b)] In this case we have $\min_{i=1,\dots,n}\ip{a_i}{y} \le \rho(A) < 0$ for all $y\in\R^m$ with $\|y\|_2 = 1$.  Thus for all $v\in \R^m$ with $\|v\|_2 \le \vert \rho(A)\vert$ we have
\[
\max_{\|y\|_2 =1}\min_{x\in\Delta_{n-1}} \ip{Ax-v}{y} \le 0.
\]
It thus follows, via a standard separation argument, that $v\in \{Ax: x\in\Delta_{n-1}\}$ for all $v\in \R^m$ with $\|v\|_2 \le \vert \rho(A)\vert$.  In particular, for each $i=1,\dots,n$ there exists $\bar x\in \Delta_{n-1}$ such that 
$A\bar x = -\vert \rho(A)\vert a_i$.  Thus $\hat x:= (\bar x + \vert \rho(A)\vert e_i)/\|\bar x + \vert \rho(A)\vert e_i\|_\infty$ satisfies \[\hat x \in L\cap \R^n_+, \; \|\hat x\|_\infty =1,\; \hat x_i \ge \vert \rho(A)\vert.
\]
Since this holds for each $i=1,\dots,n$, it follows that $\sigma(L) \ge \vert \rho(A) \vert.$  The following example shows that $\sigma(L)$ can be arbitrarily larger than $\vert \rho(A)\vert$.  Let \[A = \frac{1}{\sqrt{1+\epsilon^2}}\matr{1 & 1 & -1 & -1\\ \epsilon & -\epsilon &\epsilon & -\epsilon}\]  where $0<\epsilon< 1$.  It is easy to see that
$\vert \rho(A)\vert = \epsilon/\sqrt{1+\epsilon^2}$ and $\sigma(L) = 1$. 
\end{description}
\qed
\end{proof}

Our next result shows an analogue of Proposition~\ref{prop.compare} that relates the quantities $\rho(A)$ and $\min\{\sigma(L), \sigma(L^\perp)\}$ in the reverse direction provided the matrix $A$ is preconditioned so that its {\em rows} are orthonormal.  Note that the spaces $L = \{x\in \R^n: Ax = 0\}$ and $L^\perp = \{A\transp y: y\in \R^m\}$ are invariant if we transform $A$ via left multiplication by a non-singular matrix.  In particular, $A$ can be transformed to $\hat A:= MA$ for some non-singular $M\in \R^{m\times m}$ (e.g., via the Gram-Schmidt procedure) without changing $\sigma(L)$ and $\sigma(L^\perp)$  so that the rows of $\hat A$ are orthonormal.

\begin{proposition}\label{prop.compare.reverse}
Suppose  $A \in \R^{m\times n}$ has non-zero columns and its rows are orthonormal, that is, each row has Euclidean norm equal to one and any two different rows are orthogonal to each other. Let $L = \{x\in \R^n: Ax = 0\}$ or equivalently
$L^\perp = \{A\transp y: y\in \R^m\}$.  

\begin{description}
\item[(a)] If $L^\perp\cap \R^n_{++}\ne \emptyset$ then $\rho(A) \ge \sigma(L^\perp)/(n\sqrt{n})$.  Furthermore,
$\rho(A)$ can be arbitrarily larger than  $\sigma(L^\perp)/(n\sqrt{n})$.
\item[(b)] 
If $ L\cap \R^n_{++}\ne \emptyset$ then $\vert \rho(A)\vert \ge \sigma(L)/n^2$.  Furthermore, $\vert\rho(A)\vert$ can be arbitrarily larger than $\sigma(L)/n^2$.
\end{description}
\end{proposition}
\begin{proof}
\begin{description}
\item[(a)] Let $\bar x\in L^\perp\cap \R^n_{++}$ be such that 
$\|\bar x\|_\infty \le n$ and $\min_{i=1,\dots,n} \bar x_i \ge \sigma(L^\perp)$. 
Such $\bar x$ can be constructed by taking $\bar x = \sum_{i=1}^n x^i$ where each $x^i$ satisfies $x^i \in L^\perp \cap \R^n_+, \; \|x^i\|_\infty\le 1, \; x^i_i = \sigma_i(L^\perp).$ 

Since $L^\perp = \{A\transp y: y\in \R^m\}$ and the rows of $A$ are orthonormal, it follows that $\bar x = A\transp \bar y$ for some $\bar y\in \R^m$ with $\|\bar y\|_2 = \|\bar x\|_2 \le \sqrt{n} \|\bar x\|_\infty \le n\sqrt{n}$ and $\|a_i\|_2 \le 1$ for $i=1,\dots,n.$  Hence
\[
\rho(A) \ge \min_{i=1,\dots,n}\frac{\ip{a_i}{\bar y}}{\|\bar y\|_2\cdot \|a_i\|_2} \ge \min_{i=1,\dots,n}\frac{\bar x_i}{n\sqrt{n}\cdot \|a_i\|_2} \ge \frac{\sigma(L^\perp)}{n\sqrt{n}}.
\]
The following example shows that  $\rho(A)$ can be arbitrarily larger than  $\sigma(L^\perp)/(n\sqrt{n}).$  Let 
\[
A = \matr{\epsilon/\sqrt{2+\epsilon^2} & 1/\sqrt{2+\epsilon^2} &1/\sqrt{2+\epsilon^2}\\ 0 & -1/\sqrt{2} & 1/\sqrt{2}}
\]
where $0<\epsilon< 1$.  It is easy to see that
$\rho(A) = \sqrt{2/(4+\epsilon^2)}>\sqrt{2/5}$ and $\sigma(L^\perp) = \epsilon$. 

\item[(b)] Let $\bar x\in L\cap \R^n_{++}$ be such that 
$\|\bar x\|_\infty \le n$ and $\min_{i=1,\dots,n} \bar x_i \ge \sigma(L)$.  
Since $\|a_i\|_2\le 1$ for $i=1,\dots,n$ it suffices to show that for all $y\in \R^m$ with $\|y\|_2 = 1$ 
\[
\min_{i=1,\dots,n} \ip{a_i}{y} \le -\frac{\sigma(L)}{n^2}.
\]
To that end, let $y\in \R^m$ with $\|y\|_2=1$ be fixed and put $u:=A\transp y$. Thus $\|u\|_2 = 1$ and $u \in L^\perp$.  Let $I:=\{i\in\{1,\dots,n\}: u_i > 0\}$ and $J := \{j\in\{1,\dots,n\}: u_j < 0\}$.  Since $\bar x\in L$, $u\in L^\perp$, and $\|u\|_2 = 1$ it follows that $1= \|u\|_2^2 = \|u_I\|_2^2 + \|u_J\|_2^2$ and
\[
0 = \ip{u}{\bar x} = \ip{u_I}{\bar x_I} + \ip{u_J}{\bar x_J}.
\]
In particular, both $I,J$ must be nonempty because $\bar u \ne 0$ and $\bar x > 0$.  Since $u_J < 0,\, u_I> 0,$ and $\min_{i=1,\dots,n} \bar x_i \ge \sigma(L) > 0,$  it follows that
\[
\|\bar x_J\|_1 \cdot \min_{j\in J} u_j \le \ip{u_J}{\bar x_J} = -\ip{u_I}{\bar x_I} \le -\|u_I\|_1\cdot \sigma(L) \le - \|u_I\|_2 \cdot \sigma(L).
\]
Hence 
\[
\min_{j\in J} u_j \le 
- \frac{\sigma(L)\cdot \|u_I\|_2}{\|\bar x_J\|_1} \le - \frac{\sigma(L)\cdot \|u_I\|_2}{n\vert J\vert}\le - \frac{\sigma(L)\cdot \|u_I\|_2}{n(n-1)}.
\]
where the last two steps follow from $\|\bar x\|_\infty \le n$ and $J \ne \{1,\dots,n\}$.

On the other hand, since $u_J< 0$ and $1= \|u\|_2^2 = \|u_I\|_2^2 + \|u_J\|_2^2$ we also have 
\[
\min_{j\in J} u_j \le - \frac{\|u_J\|_2}{\sqrt{\vert J \vert}}\le - \frac{\|u_J\|_2}{\sqrt{n-1}} = - \frac{\sqrt{1-\|u_I\|_2^2}}{\sqrt{n-1}}.
\]
Therefore
\begin{align*}
\min_{i=1,\dots,n} \ip{a_i}{y} =  \min_{j\in J} u_j 
&\le \frac{1}{n(n-1)} \cdot \min_{\|u_I\| \in (0,1)} \left\{ - \sigma(L) \cdot \|u_I\|_2, - \sqrt{n^2(n-1)(1-\|u_I\|_2^2)}\right\} \\
&= -\frac{\sigma(L)}{\sqrt{n-1}\cdot \sqrt{n^2(n-1)+\sigma(L)^2}} \\
&\le -\frac{\sigma(L)}{n^2}.
\end{align*}
The following example shows that  $\vert\rho(A)\vert$ can be arbitrarily larger than  $\sigma(L)/n^2.$  Let 
\[
A = \matr{-\epsilon/\sqrt{2+\epsilon^2} & 1/\sqrt{2+\epsilon^2} &1/\sqrt{2+\epsilon^2}\\ 0 & -1/\sqrt{2} & 1/\sqrt{2}}
\]
where $0<\epsilon< 1$.  It is easy to see that
$\vert \rho(A)\vert = \sqrt{\frac{1-\sqrt{2/(4+\epsilon^2)}}{2}} > \frac{\sqrt{2-\sqrt{2}}}{2}$ and $\sigma(L) = \epsilon/2$. 

\end{description}
\end{proof}

\section{Basic procedure.}
\label{sec.bp}

This section describes an implementation of the basic procedure, which is a key component of Algorithm~\ref{algo.support}.  
To simplify notation, we describe the basic procedure for the case when $J = \{1,\dots,n\}$.  The extension to any $J \subseteq \{1,\dots,n\}$ is completely straightforward.  Suppose $P\in \R^{n\times n}$ is the projection onto a linear subspace $L\subseteq \R^n$.  The goal of the basic procedure is to  find 
either $u\in \Delta_{n-1}:=\{x\in \R^n: x\ge 0, \|x\|_1 =1 \}$ 
such that $Pu > 0$, or
$z\in \Delta_{n-1}:=\{x\in \R^n: x\ge 0, \|x\|_1 =1 \}$ such that  $\|(Pz)^+\|_1 \le \frac{1}{2}\|z\|_\infty$. To that end, consider the problem
\begin{equation}\label{eq.subproblem}
\min_{z\in \Delta_{n-1}} \frac{1}{2} \|Pz\|_2^2 \Leftrightarrow \min_{z\in \Delta_{n-1}}\max_{u\in \Delta_{n-1}} \left\{-\frac{1}{2}\|Pu\|_2^2 + \ip{Pu}{Pz}\right\}
\end{equation}
and its dual
\[
\max_{u\in \Delta_{n-1}} \left\{-\frac{1}{2}\|Pu\|_2^2 + \min_{z\in \Delta_{n-1}}\ip{Pu}{Pz}\right\} \Leftrightarrow \max_{u\in \Delta_{n-1}} \left\{-\frac{1}{2}\|Pu\|_2^2 + \min_{z\in \Delta_{n-1}}\ip{Pu}{z}\right\}. 
\]

The articles~\cite{PenaS16,PenaS19} describe several first-order  schemes for~\eqref{eq.subproblem} that achieve the goal of the basic procedure.  All of these algorithms generate sequences $z_k, u_k\in \Delta_{n-1}$ satisfying
\begin{equation}\label{eq.dual.gap}
\frac{1}{2} \|Pz_k\|_2^2+ \frac{1}{2}\|Pu_k\|_2^2 - 
\min_{z\in \Delta_{n-1}}\ip{Pu_k}{z}
 \le \mu_k
\end{equation}
for $\mu_k\rightarrow 0$.  The above property of first-order schemes is not explicitly stated in~\cite{PenaS16,PenaS19} but it can be easily inferred as shown in the recent paper~\cite{GutmP20}.

From~\eqref{eq.dual.gap} it follows that as long as $Pu_k \not > 0$ we must have $\frac{1}{2} \|Pz_k\|_2^2 \le \mu_k$.  The latter in turn implies that
\[
\|(Pz_k)^+\|_1 \le \sqrt{n} \|Pz_k\|_2 \le \sqrt{2n \mu_k} \le 
n\sqrt{2n \mu_k} \|z\|_\infty
\]
and thus the basic procedure terminates when $\mu_k \le \frac{1}{8n^3}.$  Algorithm~\ref{algo.bp} describes the {\em smooth perceptron} basic procedure which generates iterates $u_k,z_k\in\Delta_{n-1}$ satisfying~\eqref{eq.dual.gap} with $\mu_k = \frac{8}{(k+1)^2}$ and thus is guaranteed to terminate in at most $k = \Oh(n^{1.5})$ iterations.  This is both theoretically and computationally the fastest of the first-order schemes for the basic procedure proposed in~\cite{PenaS16,PenaS19}. Algorithm~\ref{algo.bp} relies on the mapping $u_{\mu}:\R^n \rightarrow \Delta_{n-1}$ defined as follows. Let $\bar u \in \Delta_{n-1}$ be fixed and $\mu > 0$. Let
\[
u_\mu(v) := \argmin_{u\in\Delta_{n-1}}\left\{\ip{u}{v} + \frac{\mu}{2}\|u-\bar u\|_2^2\right\}.
\]

{\centering\begin{minipage}{\linewidth}
\begin{algorithm}[H]
  \caption{Smooth Perceptron Scheme
    \label{algo.bp}}
  \begin{algorithmic}[1]
  \State let $u_0 := \bar u$; $\mu_0 = 2$; $z_0:=u_{\mu_0}(Pu_0);$ and $k:=0$ 
 \While {$Pu_k \ngtr 0$ and $\|(Pz_k)^+\|_1 > \epsilon\|z_k\|_\infty$}
\Statex  \quad $\theta_k:=\frac{2}{k+3}$ 
\Statex   \quad $ u_{k+1} :=(1-\theta_k)(u_k + \theta_k z_k)  + \theta_k^2 u_{\mu_k}(Pu_k)$
\Statex  \quad $\mu_{k+1} := (1-\theta_k)\mu_k$
\Statex \quad  $z_{k+1} := (1-\theta_k)z_k + \theta_k u_{\mu_{k+1}}(Pu_{k+1})$
\Statex  \quad $k := k+1$
\EndWhile

\end{algorithmic}
\end{algorithm}
\end{minipage}}

\section{A variant of Algorithm~\ref{algo.support}.}\label{sec.variant}

Algorithm~\ref{algo.support.variant} describes a
 variant of Algorithm~\ref{algo.support}
 that performs rescaling along multiple directions.  The difference between the two algorithms is the following.    
In Step 5 let $e:=(z/\|(Pz)^+)\|_1-1)^+$ and use $I+\text{Diag}(e)$ instead of $(I+e_ie_i\transp)$.  To cover the special case $(Pz)^+=0$ 
we use the following convention when $(Pz)^+ = 0$: let the $i$-th component of $e=(z/\|(Pz)^+)\|_1-1)^+$ be 
 \[
\text{$i$-th component of $e$} =  \left\{\begin{array}{rl} 0 & \text{ if } z_i = 0 \\
 +\infty & \text{ if } z_i > 0.\end{array} \right.
 \]
With this convention,   Step 5 simply trims $J$ by removing the indices corresponding to positive entries in $z$ when $(Pz)^+ = 0$.

{\centering\begin{minipage}{\linewidth}
\begin{algorithm}[H]
  \caption{Partial support with rescaling along multiple directions} \label{algo.support.variant}
  \begin{algorithmic}[1]
    \State  ({\bf Initialization}) 
    \Statex Let $D := I$, $ J := \{1,\dots,n\}$, and $\sigma\in(0,1)$ be an educated  guess of $\sigma(L)$.
       \State Let $P := P_{DL\vert J}$
	\State ({\bf Basic Procedure})
	\Statex 	\quad Find  either 
	$u\in \Delta(J)$ such that $(Pu)_J > 0$ or 
	\Statex \quad $z \in \Delta(J)$ such that $\|(Pz)^+\|_1 \le \frac{1}{2} \|z\|_\infty$.
	\State   {\bf If} $(Pu)_J > 0$ {\bf then} HALT and output  $x = D^{-1}Pu$ and $J$

\State {\bf Else (Rescale $L$ \& Trim $J$) }
\Statex \quad let $e:=\left(z/\|(Pz)^+\|_1-1\right)^+$ and $D:=(I+\text{Diag}(e))D$
\Statex \quad let $I:=\{i:D_{ii} > 1/\sigma\}$ and $J := J \setminus I$
\Statex \quad {\bf if } $J = \emptyset$ {\bf then} HALT and output  $x = 0$ and $J=\emptyset$
	\Statex \quad Go back to step 2
	\end{algorithmic}
\end{algorithm}
\end{minipage}}

\bigskip

Algorithm~\ref{algo.support.variant}, which is similar to some variants of the projection and rescaling algorithm discussed in~\cite{LourKMT16,PenaS19,Roos18}, has  convergence properties that are at least as strong as those of Algorithm~\ref{algo.support}.  The latter is an immediate consequence of the following variant of Lemma~\ref{lemma}.  We omit the proof of Lemma~\ref{lemma.2} as it is nearly identical to that of Lemma~\ref{lemma}.

\begin{lemma}\label{lemma.2} Let $L\subseteq \R^n$ be a linear subspace and $P:\R^n\rightarrow L$ be the orthogonal projection onto $L$.  Suppose $z\in\R^n_+ \setminus\{0\}$ is such that $\|(Pz)^+\|_1\le \frac{1}{2} \|z\|_\infty = \frac{1}{2}z_i$ for some $i \in \{1,\dots,n\}$.  Let $e:= (z/\|(Pz)^+\|_1-1)^+.$  Then  for $D := I+\text{Diag}(e)$ the rescaled subspace $DL\subseteq \R^n$ satisfies
\[
\sigma_i(DL) \ge 2\sigma_i(L) \text{ and } \sigma_j(DL) \ge \sigma_j(L) \text{ for } j\ne i.
\]
\end{lemma}

\section{Conclusion.}\label{sec.conclusion}

We provide a natural extension of the projection and rescaling algorithm~\cite{PenaS16, PenaS19} to find maximum support solutions to the pair of 
 feasibility problems
\[
\text{find}  \; x\in L\cap\R^n_{+} \;\;\;\; \text{ and } \; \;\;\;\;
\text{find} \; \hat x\in L^\perp\cap\R^n_{+},
\]
where $L$ is a linear subspace in $\R^n$ and $L^\perp$ is its orthogonal complement. 

Our approach hinges on three key ideas.  First, we propose a projection and rescaling algorithm that finds a point in a set of the form $L\cap \R^n_+$  that may not necessarily have  maximum support (Algorithm~\ref{algo.support}).  Second, by relying on Algorithm~\ref{algo.support} and on a key duality connection between $L\cap \R^n_+$ and $L^\perp\cap\R^n_{+}$, we propose a second algorithm that finds maximum support solutions to~\eqref{primal-dual} (Algorithm~\ref{algo.max.support}).  Third, the analyses of our algorithms rely on a novel condition measure $\min\{\sigma(L),\sigma(L^\perp)\}$ that can be seen as a refinement of condition measures previously proposed and used in~\cite{PenaS16,Ye94} (see Proposition~\ref{prop:support} and Theorem~\ref{theor}).

Our results complements the extensive and encouraging computational results reported in~\cite{PenaS19}. More precisely, we give a rigorous proof of correctness for  a minor variant of~\cite[Algorithm 1]{PenaS19}.




%
%
%






\bibliographystyle{plain}

\end{document}